\documentclass[11pt]{article}
\usepackage{amsmath, amsthm, amssymb}
\usepackage[mathscr]{eucal}
\usepackage{mathrsfs}
\usepackage{psfrag}
\usepackage{fullpage}
\usepackage{color}
\usepackage{eucal}
\usepackage[dvips]{graphicx}
\usepackage{multirow}
\usepackage{fancyhdr}
\usepackage{enumerate}

\newtheorem{proposition}{Proposition}[section]

\newtheorem{remark}[proposition]{Remark}

\newtheorem{prop}[proposition]{Proposition}
\newtheorem{thm}[proposition]{Theorem}
\newtheorem{cor}[proposition]{Corollary}
\newtheorem{lem}[proposition]{Lemma}

\newtheorem{rem}[proposition]{Remark}

\newtheorem{ass}[proposition]{Assumption}

\numberwithin{equation}{section}
\numberwithin{proposition}{section}

\begin{document}

\title{Large Deviations Principle for a Large Class of One-Dimensional  Markov Processes}
\author{Konstantinos Spiliopoulos
\footnote{Lefschetz Center for
Dynamical Systems, Division of Applied Mathematics,
Brown University, Providence, RI,
02912, USA,E-mail: kspiliop@dam.brown.edu}
}
\date{}
\maketitle
%
%
%

\begin{abstract} We study the large deviations principle for one dimensional,
continuous, homogeneous, strong Markov processes that do not
necessarily behave locally as a Wiener process. Any strong Markov
process $X_{t}$ in $\mathbb{R}$ that is continuous with probability
one, under some minimal regularity conditions, is governed by a
generalized elliptic operator $D_{v}D_{u}$, where $v$ and $u$ are
two strictly increasing functions, $v$ is right continuous and $u$
is continuous. In this paper, we study large deviations principle for Markov processes
whose infinitesimal generator is $\epsilon D_{v}D_{u}$ where
$0<\epsilon\ll 1$. This result generalizes the classical large
deviations results for a large class of one dimensional "classical"
stochastic processes. Moreover, we consider reaction-diffusion
equations governed by a generalized operator $D_{v}D_{u}$. We apply
our results to the problem of wave front propagation for these type
of reaction-diffusion equations.

\vspace{0.1cm}

\noindent \textit{Key words}: Large deviations principle, Action functional, Strong Markov processes in one dimension, Wave front propagation, Reaction - diffusion equations.

\vspace{0.1cm}

\noindent \textit{Mathematics Subject Classification (2000)}: Primary 60F10, 60J60; secondary 60G17.
\end{abstract}

\section{Introduction}
It is well known that for each classical second order differential operator
\begin{equation}
Lf(x)=\frac{1}{2}a(x) \frac{d^{2}f(x)}{dx^{2}}+b(x)\frac{df(x)}{dx}\label{SmoothOperator}
\end{equation}
with smooth enough coefficients $a(x)>0$ and $b(x)$, there exists a
diffusion process $(X_{t},\mathbb{P}_{x})$ in $\mathbb{R}$ such that
$L$ is the generator of this process. The domain of definition of
$L$ is $\mathcal{D}(L)=\lbrace f: f\in
\mathcal{C}^{2}(\mathbb{R})\rbrace$. If $a(x), b(x) \in
\mathcal{C}(\mathbb{R})$ with $a(x)>0$, the trajectories of $X_{t}$
can be constructed as the solutions of the following stochastic
differential equation:
\begin{equation}
dX_{t}=\sigma(X_{t})dW_{t}+ b(X_{t})dt, \hspace{0.1cm} X_{0}=x,
\label{SDE1}
\end{equation}
where $a(x)=\sigma^{2}(x)$ and $W_{t}$ is the standard Wiener
process in $\mathbb{R}$. It is also widely known that if $X_{t}$
satisfies (\ref{SDE1}) then it behaves locally like a Wiener
process. In particular, it spends zero time at any given point
$x\in\mathbb{R}$ and it exits the interval $[x-\delta,x+\delta]$
through both ends with asymptotically equal probabilities as
$\delta\downarrow 0$.

Let now $0<\epsilon\ll 1$ be a small positive number. Denote by $X^{\epsilon}_{t}$ the process that is governed by the operator
\begin{equation}
L^{\epsilon}f(x)=\frac{\epsilon}{2}a(x) \frac{d^{2}f(x)}{dx^{2}}+b(x)\frac{df(x)}{dx}.\label{SmoothOperator1}
\end{equation}

Then, large deviations principle for the process $X^{\epsilon}_{t}$ is well known (Freidlin and Wentzel \cite{FW5}; see also \cite{F1} and \cite{FW1}). 
In particular, the action functional  for the process
$\left(X^{\epsilon}_{t}\right)_{t\in[0,T]}$, in $\mathcal{C}([0,T];\mathbb{R})$ as $\epsilon
\downarrow 0$ has the form $\frac{1}{\epsilon}S_{0T}(\phi)$, where
\begin{eqnarray}
 S_{0T}(\phi)=\begin{cases}\frac{1}{2}\int_{0}^{T}\frac{|\dot{\phi}_{s}-b(\phi_{s})|^{2}}{a(\phi_{s})}ds, & \text{if } \phi\in \mathcal{C}([0,T];\mathbb{R}) \text{ is absolutely continuous}
 \\
               +\infty, & \text{for the rest of } \mathcal{C}([0,T];\mathbb{R}).
       \end{cases}  \label{ActionFunctionalSmooth}
\end{eqnarray}

However, no general results on large deviations principle are known
for general one-dimensional, strong Markov processes that do
not behave
locally as a Wiener process. Namely, for processes that may spend positive time at a given point $x\in\mathbb{R}$ or that may exit a given interval $[x-\delta,x+\delta]$ with unequal probabilities from left and right as $\delta\downarrow 0$.
The purpose of this paper is to study exactly this situation  for a
large class of one dimensional, homogeneous, strong Markov processes that are
continuous with probability one. These processes were characterized
by Feller \cite{Feller} in a unique way through a generalized second
order elliptic operator $D_{v}D_{u}$ and its domain of definition.

As we shall also see below, the functions $v$ and $u$ that appear in the $D_{v}D_{u}$ operator  are in general non smooth. Function $u$ could be non differentiable and function $v$ could even have jump discontinuities. Note that if they were sufficiently smooth, then one would recover the classical second order operator (\ref{SmoothOperator1}) (see below for more details). These non-smoothness issues create several technical difficulties in the proof of the large deviations principle that one has to overcome. We overcome these difficulties and we provide an explicit expression for the action functional which is in terms of the $u$ and the $v$ functions under minimal assumptions on $u$ and $v$. Moreover, we apply our results to the problem of wave front propagation for reaction diffusion equations where the operator of the partial differential equation is a generalized elliptic operator $D_{v}D_{u}$. Such reaction diffusion equations can appear in applications as, for example, the limit of a family of standard reaction-diffusion equations where the diffusion and drift coefficients converge to non-smooth functions. Then, as we shall also see in section 4, the characterization of the limit through a $D_{v}D_{u}$ operator is very convenient and one can use the expression for the action functional to calculate the position of the wave front. Moreover, the non-smoothness of the $v$ and $u$ functions can create several phenomena in the propagation of the front such as change in the speed of the propagation.

In addition, such $D_{v}D_{u}$ processes arise naturally in applications as limits of
diffusion processes. For example, we mention: $(a)$ the limiting process
for nondegenerate diffusion in narrow branching tubes with
reflection at the boundary (see Freidlin and Wentzel \cite{FW3}) and $(b)$ the Wiener process with reflection in non-smooth narrow tubes (see Spiliopoulos \cite{SpilEJP2009}). In both cases, the diffusion process in the narrow branching tube or in the narrow non-smooth tube (for $(a)$ and $(b)$
respectively) converges weakly to a strong Markov process $X_{t}$, as the tube
becomes thinner and thinner. The limiting process behaves like a
standard diffusion process on the left and on the right of the point
where the branching occurs or of the discontinuity point (for $(a)$ and $(b)$ respectively) and has to satisfy a gluing condition at
that point.  Knowing the action functional for these kind of
processes, one can study several other problems of interest. We
mention, for example: $(i)$ exit problems, $(ii)$ wave front
propagation for reaction diffusion equations where the operator of
the partial differential equation is a generalized elliptic operator
$D_{v}D_{u}$ and other related problems.

In this paper we study the large deviations principle for a one
dimensional strong Markov process $X^{\epsilon}_{t}$ with generator
$\epsilon D_{v}D_{u}$, where $u(x)$ and $v(x)$ are given functions, and $X^{\epsilon}_{0}=x$.
In particular, $u(x)$ and $v(x)$ are strictly increasing functions,
$u(x)$ is continuous and $v(x)$ is right continuous and $D_{v}$,
$D_{u}$ are differentiation operators with respect to $v$ and $u$
respectively. The expression for the action functional is in Theorem
\ref{MainTheorem}. Corollary \ref{MainCorollary1} gives an
equivalent and simpler expression for the action functional under
some stricter assumptions. These results generalize the classical
large deviations results for a large class of one dimensional strong
Markov processes that cannot be expressed as solutions to stochastic
differential equations. In particular, Corollary
\ref{MainCorollary2} shows that our form of the action functional
reduces to (\ref{ActionFunctionalSmooth}) with $b=0$, if $u$ and $v$
have a special form and enough smoothness is provided.

Before mentioning the main result of this paper (Theorem \ref{MainTheorem}) we need to introduce some notation. Let us define the sets
\begin{eqnarray}
U&=&\lbrace x\in\mathbb{R}:\textrm{ the derivative of } u \textrm{ does not exist at }x \rbrace \nonumber\\
V&=&\lbrace x\in\mathbb{R}:\textrm{ the derivative of } v \textrm{ does not exist at }x, v\textrm{ is continuous at }x \rbrace \nonumber\\
V_{d}&=&\lbrace x\in\mathbb{R}: v \textrm{ is discontinuous at } x
\rbrace \label{UandVfunctionsPoints}
\end{eqnarray}
Of course, the sets $U,V$ and $V_{d}$ are at most countably infinite.

Moreover, for a continuous function $\phi:[0,T]\rightarrow
\mathbb{R}$, i.e. $\phi \in \mathcal{C}([0,T];\mathbb{R})$, we define the sets
\begin{eqnarray}
U_{\phi}&=&\lbrace t\in[0,T]:\phi(t)\in U \rbrace \nonumber\\
V_{\phi}&=&\lbrace t\in[0,T]:\phi(t)\in V \rbrace \nonumber\\
V_{d,\phi}&=&\lbrace t\in[0,T]:\phi(t)\in V_{d} \rbrace. \label{Phi_functionPoints}
\end{eqnarray}
We also define the sets
\begin{equation}
E=(U\cup V) \setminus V_{d} \hspace{0.2cm} \textrm{ and }\hspace{0.2cm} E_{\phi}=(U_{\phi}\cup V_{\phi}) \setminus V_{d,\phi}.\label{Phi_functionPoints2}
\end{equation}
Now we are ready to state the main result of this paper.
\begin{thm}
Let  $u(x)$ and $v(x)$ be strictly increasing functions, $u(x)$ be
continuous and $v(x)$ be right continuous. Assume that
there are positive constants $c_{1}$ and $c_{2}$ such that $0<
u'(x)\leq c_{1}$ and $0<c_{2}\leq v'(x)$ at the points $x$ where the
derivatives of $u(x)$ and $v(x)$ exist. Let $X^{\epsilon}_{t}$ be the strong Markov process
whose infinitesimal generator is $\epsilon D_{v}D_{u}$ for
$0<\epsilon\ll 1$ with initial point $X^{\epsilon}_{0}=x$.

Let $\phi:[0,T]\rightarrow \mathbb{R}$ be a continuous function in
$[0,T]$. We have the following.
\begin{enumerate}
 \item If the Lebesgue measure of the set $E_{\phi}$ is zero, i.e. $\Lambda (E_{\phi})=0$,
 then
 \begin{equation}
         \sigma_{\phi}(t)=\int_{0}^{t}[\frac{1}{2}\frac{dv}{du}(\phi_{s})]^{-1}ds\label{ChangeTimeWithoutN_1}
 \end{equation}
 is well defined, it is continuous and
 non-decreasing in $t$. If  $\Lambda (V_{d,\phi})=0$, then
 $\sigma_{\phi}(t)$ is strictly increasing in $t$. For functions $\phi$  such that $\Lambda (E_{\phi})>0$ we interpret, without loss of generality,
 the derivative $\frac{dv}{du}$ in the formula for $\sigma_{\phi}(t)$ as the minimum of the left and right derivatives of $v$ with respect to $u$ on the countable set $E$ 
(see Remark \ref{R:DiscontinuityPoints} and the statement of Lemma \ref{LemmaTimeChange1} for more details).
 \item Denote by $\gamma_{\phi}(t)$ the generalized inverse to $\sigma_{\phi}(t)$, i.e.
         \begin{equation}
         \gamma_{\phi}(t)=\inf \lbrace s: \sigma_{\phi}(s)>t\rbrace.\label{ChangeTimeWithoutN_2}
        \end{equation}
   The action functional  for the process
$\left(X^{\epsilon}_{t}\right)_{t\in[0,T]}$, in $\mathcal{C}([0,T];\mathbb{R})$ as $\epsilon
\downarrow 0$ has the form $\frac{1}{\epsilon}S_{0T}(\phi)$ where
\begin{eqnarray}
S_{0T}(\phi)= \begin{cases}\frac{1}{2}\int_{0}^{\sigma_{\phi}(T)}|\frac{d
u(\phi(\gamma_{\phi}(s)))}{ds}|^{2}ds, & \text{if }
u(\phi(\gamma_{\phi}(s))) \text{ is absolutely continuous and } \phi_{0}=x\\  & \\
                                                     +\infty, & \text{for the rest of }  \mathcal{C}([0,T];\mathbb{R}).   \end{cases}\label{ActionFunctional}
\end{eqnarray}
The functional
$S_{0T}(\phi)$ is lower semi-continuous in the sense of uniform
convergence. Namely, if a sequence $\phi^{n}$ converges uniformly to
$\phi$ in $\mathcal{C}([0,T];\mathbb{R})$, then $S_{0T}(\phi)\leq
\liminf_{n\rightarrow \infty}S_{0T}(\phi^{n})$. Lastly, the set $\Phi_{s}=\lbrace \phi \in \mathcal{C}([0,T];\mathbb{R}):
S_{0T}(\phi)\leq s \text{ and } \phi(0) \text{ belongs to a compact subset of } \mathbb{R} \rbrace $ is compact.
\end{enumerate}
\begin{flushright}
$\square$
\end{flushright}
\label{MainTheorem}
\end{thm}

The following corollary gives a useful representation of the action
functional in the case where $v$ is a continuous function. Then, of
course, $V_{d}=\emptyset$, $E=U\cup V$ and $\sigma_{\phi}(t)$ is strictly
increasing. It follows directly from Theorem \ref{MainTheorem} after
a straightforward change of variables.

\begin{cor}
In addition to the assumptions of Theorem \ref{MainTheorem}, let us
assume that the function $v(x)$ is continuous. The action functional
for the process $\left(X^{\epsilon}_{t}\right)_{t\in[0,T]}$, in
$\mathcal{C}([0,T];\mathbb{R})$ as $\epsilon \downarrow 0$ is
$\frac{1}{\epsilon}S_{0T}(\phi)$ where
\begin{eqnarray}
S_{0T}(\phi)= \begin{cases}\frac{1}{4}\int_{0}^{T}(u  \circ \phi)'(s)(v  \circ \phi)'(s)ds, & \text{if }
\phi \text{ is absolutely continuous and } \phi_{0}=x\\ & \\
                                                     +\infty, & \text{for the rest of }\mathcal{C}([0,T];\mathbb{R}).   \end{cases}\label{ActionFunctionalreduced}
\end{eqnarray}\label{MainCorollary1}
Moreover, note that for $\phi$ absolutely continuous we have $S_{E_{\phi}}(\phi)=0$.
\begin{flushright}
$\square$
\end{flushright}
\end{cor}

\begin{remark}\label{R:DiscontinuityPoints}
As we saw in the statement of Theorem \ref{MainTheorem} part (i), $\sigma_{\phi}(t)$ is well defined for $\phi$ such that $\Lambda (E_{\phi})=0$.
 As a consequence, the action functional is also well defined. For $\phi$ such that $\Lambda (E_{\phi})>0$ we defined $\sigma_{\phi}(t)$ using formula 
(\ref{ChangeTimeWithoutN_1}) by interpreting the derivative $\frac{dv}{du}$  as the minimum of the left and right derivatives of $v$ with respect to $u$ on the countable set $E$. 
This is done without loss of generality. In particular, let us pick a point $z\in \left(U \cup V\right)\setminus V_{d}$ and denote $E^{z}_{\phi}=\{t\in[0,T]: \phi_{t}=z\}$.
 Then, for $\phi$ absolutely continuous, we have $S_{E^{z}_{\phi}}(\phi)=0$ (independently of the interpretation of the $u$ and $v$ derivatives on $E$). 
More details will be given in the proof of Theorem \ref{MainTheoremHelp1}.
\begin{flushright}
$\square$
\end{flushright}
\end{remark}

For the convenience of the reader, we briefly recall the Feller
characterization of all one-dimensional Markov processes, that are
continuous with probability one (for more details see \cite{Feller};
also \cite{M1}). All one-dimensional strong Markov processes that
are continuous with probability one, can be characterized (under
some minimal regularity conditions) by a generalized second order
differential operator $D_{v}D_{u}f$ with respect to two increasing
functions $u(x)$ and $v(x)$; $u(x)$ is continuous, $v(x)$ is right
continuous. In addition, $D_{u}$, $D_{v}$ are differentiation
operators with respect to $u(x)$ and $v(x)$ respectively, which are
defined as follows:

$D_{u}f(x)$ exists if $D_{u}^{+}f(x)=D_{u}^{-}f(x)$, where the left
derivative of $f$ with respect to $u$ is defined as follows:
\begin{displaymath}
D_{u}^{-}f(x)=\lim_{h\downarrow 0}\frac{f(x-h)-f(x)}{u(x-h)-u(x)}
\hspace{0.2cm} \textrm{ provided the limit exists.}
\end{displaymath}
The right derivative $D_{u}^{+}f(x)$ is defined similarly. If $v$ is
discontinuous at $y$ then
\begin{displaymath}
 D_{v}f(y)=\lim_{h\downarrow 0}\frac{f(y+h)-f(y-h)}{v(y+h)-v(y-h)}.
\end{displaymath}
\begin{rem}\label{R:ReductionToSimpleCase}
For example, it is easy to see that the operator $L$  in
(\ref{SmoothOperator}) can be written as a $D_{v}D_{u}$ operator
with $u$ and $v$ as follows:
\begin{equation}
u(x)=\int_{0}^{x} e^{-\int_{0}^{y}\frac{2b(z)}{a(z)}dz} dy
\hspace{0.3cm}\textrm{ and }\hspace{0.3cm}
v(x)=\int_{0}^{x}\frac{2}{a(y)} e^{\int_{0}^{y}\frac{2b(z)}{a(z)}dz}
dy. \label{UandVfunctions}
\end{equation}
The representation of $u(x)$ and $v(x)$ in (\ref{UandVfunctions}) is
unique up to multiplicative and additive constants. In fact, one can
multiply one of these functions by some constant and divide the
other function by the same constant or add a constant to either of
them.
\begin{flushright}
$\square$
\end{flushright}
\end{rem}
Corollary \ref{MainCorollary2} is easily obtained from  Corollary \ref{MainCorollary1} and Remark \ref{R:ReductionToSimpleCase}. It shows in which way the action functional 
in (\ref{ActionFunctionalSmooth})
 is
generalized by the functional in (\ref{ActionFunctional}) in the case of $b=0$.
\begin{cor}
If $u(x)$ and $v(x)$ are given by
(\ref{UandVfunctions}) and $a(x)$, $b(x)$ are regular enough, then
$E_{\phi}=\emptyset$ and the action functional in
(\ref{ActionFunctional}), or equivalently in (\ref{ActionFunctionalreduced}), coincides with
(\ref{ActionFunctionalSmooth}) with $b=0$.\label{MainCorollary2}
\begin{flushright}
$\square$
\end{flushright}
\end{cor}

The rest of the paper is organized as follows. In section 2, we
prove that (\ref{ActionFunctional}) is the action functional for
$\left(X^{\epsilon}_{t}\right)_{t\in[0,T]}$ assuming that
(\ref{ChangeTimeWithoutN_1}) is well defined. In section 3, we
prove: (a) that $\sigma_{\phi}(t)$ in (\ref{ChangeTimeWithoutN_1})
is well defined for functions $\phi$ such that the Lebesgue measure
of the set $E_{\phi}$ is zero and (b) several auxiliary results that
are used in section 2 to prove Theorem \ref{MainTheorem}. In section
4, we consider reaction-diffusion equations governed by a
generalized operator $D_{v}D_{u}$ and we apply our results to the
problem of wave front propagation for these type of
reaction-diffusion equations.
Lastly, section 5 includes some concluding comments and remarks on
future work.
\section{Estimates for probabilities of large deviations}

In this section we prove that (\ref{ActionFunctional}) is the action functional for $\left(X^{\epsilon}_{t}\right)_{t\in[0,T]}$.
However, first we introduce some notation that we will use throughout the paper and we state the results of \cite{Volkonskii1} that we use. Then we state without proof some auxiliary results. The proof of these auxiliary lemmas will be given in the next section.

In this and the following sections we will denote by $C_{0}$ any
unimportant constants that do not depend on any small or big parameter. The
constants may change from place to place though, but they will
always be denoted by the same $C_{0}$. 
Moreover, we fix two functions $u(x)$ and $v(x)$ that have the properties of Theorem \ref{MainTheorem} and we denote by $X^{\epsilon}_{t}$ for the 
process whose infinitesimal generator is $\epsilon DvDu$. Additionally, let $u_{-1}(x)$ denoting the inverse function of $u(x)$.

Furthermore, for a continuous function $\phi:[0,T]\rightarrow \mathbb{R}$ we define the functions $\sigma_{u_{-1}(\phi)}(t)$ and $\gamma_{u_{-1}(\phi)}(t)$ in the same way to (\ref{ChangeTimeWithoutN_1}) and (\ref{ChangeTimeWithoutN_2}) with $u_{-1}(\phi)$ in place of $\phi$.

The following key result is a restatement of Theorem 4 in \cite{Volkonskii1}.

\begin{thm}
Let  $u(x)$ and $v(x)$ be strictly increasing functions, $u(x)$ be continuous and $v(x)$ be right continuous.
Let $\left(v_{n}(x)\right)_{n\in\mathbb{N}}$ be a sequence of strictly increasing functions, continuously differentiable with respect to $u(x)$ and converging to $v(x)$ at every continuity point of $v(x)$. Moreover,  $W_{t}$ denotes the standard one dimensional Wiener process.

We introduce the variables $\tau^{n}_{u_{-1}(W)}(t)$ by the equations
\begin{equation}
\int^{\tau^{n}_{u_{-1}(W)}(t)}_{0}\frac{1}{2}\frac{dv_{n}}{du}(u_{-1}(W_{s}))ds=t
\label{RandomChangeTimeWithN_1}
\end{equation}
Then we have:
\begin{enumerate}
\item  $\lim_{n\rightarrow \infty}\tau^{n}_{u_{-1}(W)}(t)$ exists uniformly in $t\geq 0$ on any finite time interval in the sense of convergence in probability, 
for all measures $\mathbb{P}_{x}$ and independently of the choice of the sequence
 $\left(v_{n}\right)_{n\in\mathbb{N}}$. Moreover,  $\lim_{n\rightarrow \infty}\tau^{n}_{u_{-1}(W)}(t)$ is strictly increasing in $t$ with $\mathbb{P}_{x}$ probability $1$.
 \item Denote
        \begin{equation}
         \tau_{u_{-1}(W)}(t)=\lim_{n\rightarrow \infty}\tau^{n}_{u_{-1}(W)}(t)\label{RandomChangeTimeWithoutN_1}
        \end{equation}
       The process
      \begin{equation}
        X_{t}=u_{-1}[W_{\tau_{u_{-1}(W)}(t)}]\label{InTermsOfWienerProcess1}
      \end{equation}
        is a homogeneous, strong Markov process whose infinitesimal generator is $D_{v}D_{u}$. The domain of definition of the $D_{v}D_{u}$ operator is
   \begin{eqnarray}
     \mathcal{D} (D_{v}D_{u})&=&\lbrace f: f\in \mathcal{C}_{c}(\mathbb{R})\textrm{, where at each non smoothness point }\nonumber\\
&  &x_{i} \textrm{ of } u \textrm{ and } v \textrm{ the gluing condition holds}\label{DomainOfDefinition1}\\
& &D_{u}^{+}f(x_{i})-D_{u}^{-}f(x_{i})=[v(x_{i}+)-v(x_{i}-)] DvDuf(x_{i}) \nonumber\\
& &  \textrm{and } DvDuf(x_{i})=\lim_{x\rightarrow
x_{i}^{+}}DvDuf(x)=\lim_{x\rightarrow x_{i}^{-}}DvDuf(x) \rbrace.
\nonumber
\end{eqnarray}
\end{enumerate}\label{VolkonskiiTheorem1}
\begin{flushright}
$\square$
\end{flushright}
\end{thm}

\begin{rem}
Theorem \ref{VolkonskiiTheorem1} essentially says that any
continuous, homogeneous, strong Markov process that can be
characterized through a $D_{v}D_{u}$ operator, can be obtained from
a Wiener process after a random time change and a space
transformation. Moreover, a simple application of It\^{o} formula
shows that if $u(x)$ and $v(x)$ are given by (\ref{UandVfunctions})
and $a(x), b(x)$ are regular enough, then
$X_{t}=u_{-1}[W_{\tau_{u_{-1}(W)}(t)}]$ satisfies (\ref{SDE1}).
\begin{flushright}
$\square$
\end{flushright}
\end{rem}

We will also need the following results whose proof will be given in
the next section.  Lemma \ref{LemmaTimeChange1} is essentially part (i) of Theorem \ref{MainTheorem}. Lemmas \ref{LemmaTimeChange1a} and \ref{LemmaTimeChange2} are technical lemmas that will be used in the proof of lower semicontinuity of the functional $S_{0T}(\phi)$ and compactness of the set $\Phi_{s}=\lbrace \phi \in \mathcal{C}([0,T];\mathbb{R}):
S_{0T}(\phi)\leq s \rbrace $. Proposition \ref{Proposition21} gives a
representation of the process $X^{\epsilon}_{t}$ that is governed by
the generator $\epsilon D_{v}D_{u}$ in the spirit of Theorem
\ref{VolkonskiiTheorem1}. Lemma
\ref{ExponentiallyTight} discusses the exponential tightness of
$Y^{\epsilon}_{t}=u(X^{\epsilon}_{t})$. Using the aforementioned results we prove Theorems
\ref{MainTheoremHelp2} and \ref{MainTheoremHelp1} which discuss the large
deviations principle for $Y^{\epsilon}_{t}=u(X^{\epsilon}_{t})$.

The proof of Theorem \ref{MainTheorem} follows from Remark \ref{R:DiscontinuityPoints}, Theorems
\ref{MainTheoremHelp2} and \ref{MainTheoremHelp1} and
the well known contraction principle for large deviations. Namely, we find the action functional of $X^{\epsilon}_{t}$
by using the action functional for $Y^{\epsilon}_{t}$ and the fact that $u(x)$ is invertible.

\begin{lem}
Let  $u(x)$ and $v(x)$ be strictly increasing functions as in Theorem \ref{MainTheorem}. In addition, let $\left(v_{n}(x)\right)_{n\in N}$ be a sequence of strictly increasing functions,
continuously differentiable with respect to $u(x)$ and converging to $v(x)$ at every continuity point of $v(x)$.
Moreover, assume that $0<c_{2}\leq v'_{n}(x)$ for every $n$.

Let $\phi:[0,T]\rightarrow \mathbb{R}$ be a continuous function in
$[0,T]$, i.e. $\phi \in \mathcal{C}([0,T];\mathbb{R})$. We introduce the
functions $\sigma^{n}_{\phi}(t)$ by the formula
\begin{equation}
\sigma^{n}_{\phi}(t)=\int_{0}^{t}[\frac{1}{2}\frac{dv_{n}}{du}(\phi_{s})]^{-1}ds
\label{ChangeTimeWithN_1}
\end{equation}
The functions $\sigma^{n}_{\phi}(t)$ can be regarded as functions of
$t$ or as functionals of $\phi$. If  $\Lambda (E_{\phi})=0$ then
$\lim_{n\rightarrow \infty}\sigma^{n}_{\phi}(t)$ exists uniformly in
$t$ on any finite time interval and independently of the choice of
the sequence $\left(v_{n}\right)_{n\in N}$. Moreover, it is continuous and non-decreasing
in $t$. If $\Lambda (V_{d,\phi})=0$, then $\lim_{n\rightarrow
\infty}\sigma^{n}_{\phi}(t)$ is strictly increasing in $t$. We write
\begin{equation}
\sigma_{\phi}(t)=\int_{0}^{t}[\frac{1}{2}\frac{dv}{du}(\phi_{s})]^{-1}ds=\lim_{n\rightarrow
\infty}\sigma^{n}_{\phi}(t).\label{ChangeTimeWithoutN_11}
\end{equation}\label{LemmaTimeChange1}
\begin{flushright}
$\square$
\end{flushright}
\end{lem}

\begin{lem}
Let $\phi:[0,T]\rightarrow \mathbb{R}$ be a continuous function in
$[0,T]$ such that $\sigma_{\phi}(t)$ is well defined for $t\in[0,T]$. Function $\gamma_{\phi}(t)$ is right continuous.
Let us define $\gamma_{\phi}^{-}(t)=\lim_{s\rightarrow
t^{-}}\gamma_{\phi}(s)$. For any $t\in
[0,\sigma_{\phi}(T)]$ that is not a continuity point of
$\gamma_{\phi}(t)$,  the function $\phi(s)$ is constant for
$s\in[\gamma_{\phi}^{-}(t),
\gamma_{\phi}(t)]$.\label{LemmaTimeChange1a}
\begin{flushright}
$\square$
\end{flushright}
\end{lem}

\begin{lem}
Let $\phi^{n}$ be a sequence of functions in $\mathcal{C}([0,T];\mathbb{R})$ that converges to $\phi$ uniformly in $\mathcal{C}([0,T];\mathbb{R})$. Under the assumptions of Theorem \ref{MainTheorem} for the functions $v$ and $u$ we have:
\begin{enumerate}
 \item For any $t\in[0,T]$ we have that $\sigma_{\phi}(t)=\lim_{n\rightarrow \infty}\sigma_{\phi^{n}}(t)$. The convergence holds uniformly in $t$.
 \item For any $t\in [0,\sigma_{\phi}(T)]$ that is a continuity point of $\gamma_{\phi}(t)$ we
       have $\gamma_{\phi}(t)=\lim_{n\rightarrow \infty}\gamma_{\phi^{n}}(t)$.
 \item For any $t\in [0,\sigma_{\phi}(T)]$ we have $\phi(\gamma_{\phi}(t))=\lim_{n\rightarrow \infty}\phi^{n}(\gamma_{\phi^{n}}(t))$.
\end{enumerate}\label{LemmaTimeChange2}
\begin{flushright}
$\square$
\end{flushright}
\end{lem}

\begin{prop}
Let us define $X^{\epsilon}_{t}=u_{-1}[\sqrt{\epsilon}W_{\tau_{u_{-1}(\sqrt{\epsilon}W)}(t)}]$, where $\tau_{u_{-1}(\sqrt{\epsilon}W)}(t)$ is defined as in
 (\ref{RandomChangeTimeWithoutN_1}) with $\sqrt{\epsilon}W$ in place of $W$. Then, the infinitesimal generator of $X^{\epsilon}_{t}$ is $\epsilon DvDu$.\label{Proposition21}
\begin{flushright}
$\square$
\end{flushright}
\end{prop}

\begin{lem}
The family
$Y^{\epsilon}_{t}=\sqrt{\epsilon}W_{\tau_{u_{-1}(\sqrt{\epsilon}W)}(t)}$,
is exponentially tight in $\mathcal{C}([0,T];\mathbb{R})$: for
any $\alpha>0$  and $\delta>0$ there exists a compact
$K_{\alpha}\subset \mathcal{C}([0,T];\mathbb{R})$ such that

$$\mathbb{P}(\rho_{0T}(Y^{\epsilon}_{\cdot}, K_{a})\geq \delta)<\exp \lbrace -\frac{a}{\epsilon}\rbrace
$$ for $\epsilon>0$ small enough.\label{ExponentiallyTight}
\begin{flushright}
$\square$
\end{flushright}
\end{lem}

\begin{rem}
In what follows we will use  Lemmas (\ref{LemmaTimeChange1a}) and (\ref{LemmaTimeChange2}) with $\phi=u_{-1}(\psi)$, where $\psi$ is a continuous function.
\begin{flushright}
$\square$
\end{flushright}
\end{rem}

Let us define now the functional
\begin{eqnarray}
S_{0T}^{Y}(\psi)=\begin{cases}\frac{1}{2}\int_{0}^{\sigma_{u_{-1}(\psi)}(T)}|\frac{d \psi(\gamma_{u_{-1}(\psi)}(s))}{ds}|^{2}ds, & \text{if }  \psi(\gamma_{u_{-1}(\psi)}(s)) \text{is absolutely continuous and } \psi_{0}=u(x)
\\&\\
                                                     +\infty, & \text{for the rest of }\mathcal{C}([0,T];\mathbb{R}). \end{cases}\label{ActionFunctional2}
\end{eqnarray}
Remark \ref{R:DiscontinuityPoints}, Theorems \ref{MainTheoremHelp2} and \ref{MainTheoremHelp1} below imply that the
action functional for the process $\left(Y^{\epsilon}_{t}\right)_{t\in[0,T]}$ on
$\mathcal{C}([0,T];\mathbb{R})$ as $\epsilon\downarrow 0$ is given by
$\frac{1}{\epsilon}S_{0T}^{Y}(\psi)$. Theorem \ref{MainTheoremHelp2} discusses the standard properties of
$S_{0T}^{Y}(\psi)$. In particular,  $S_{0T}^{Y}(\psi)$ is lower semi-continuous in the sense of
uniform convergence and the set $\Psi_{s}=\lbrace \psi
\in \mathcal{C}([0,T];\mathbb{R}): S_{0T}^{Y}(\psi)\leq s \rbrace $ is compact. Theorem \ref{MainTheoremHelp1} is about the
estimates for probabilities of large deviations.  Then, as we mentioned before,
Theorem \ref{MainTheorem} follows from these two theorems, Remark \ref{R:DiscontinuityPoints} and the well known contraction principle for large deviations.

\begin{thm}
Let $u$ and $v$ be two strictly increasing functions as in Theorem
\ref{MainTheorem} and let $S_{0T}^{Y}(\psi)$  be defined by
(\ref{ActionFunctional2}). Then
\begin{enumerate}
 \item  The functional $S_{0T}^{Y}(\psi)$ is lower semi-continuous in the sense of uniform convergence.
   Namely, if a sequence $\psi^{n}$ converges uniformly to $\psi$ in $\mathcal{C}([0,T];\mathbb{R})$, then $S_{0T}^{Y}(\psi)\leq \liminf_{n\rightarrow \infty}S_{0T}^{Y}(\psi^{n})$.
 \item The set 
$\Psi_{s}=\lbrace \psi \in \mathcal{C}([0,T];\mathbb{R}): S_{0T}^{Y}(\psi)\leq s \text{ and } \psi(0) \text{ belongs to a compact subset of } \mathbb{R} \rbrace $ is compact.
 \end{enumerate} \label{MainTheoremHelp2}
 \begin{flushright}
$\square$
\end{flushright}
\end{thm}
\begin{proof}
(i). It is sufficient to consider the case when $S_{0T}^{Y}(\psi^{n})$ has a finite limit. The proof follows directly from Lemma \ref{LemmaTimeChange2} and the fact that
$\psi(\gamma_{u_{-1}(\psi)}(s))$ is absolutely continuous (see
\cite{RieszNagy} page 75 and the proof of the corresponding property for the action functional of the Wiener process \cite{FW5}).

(ii). Let $\psi\in \Psi_{s}$, i.e. $S_{0T}^{Y}(\psi)\leq s$. It is
enough to prove that
 \begin{enumerate}[a)]
 \item $|\psi(t)|\leq C_{0}<\infty$ for some constant $C_{0}$ uniformly in  $t\in[0,T]$.
  \item $|\psi(t+h)-\psi(t)|\leq g(h)\rightarrow 0$ as $h\rightarrow 0$ for some function $g(h)$ uniformly in $t\in[0,T]$.
\end{enumerate}
Then we can conclude by the well known Ascoli-Arzela theorem.

We have two cases: $\gamma_{u_{-1}(\psi)}(\cdot)$ is continuous at
$\sigma_{u_{-1}(\psi)}(t)\in[0,\sigma_{u_{-1}(\psi)}(T)]$ and
$\gamma_{u_{-1}(\psi)}(\cdot)$ is not continuous at
$\sigma_{u_{-1}(\psi)}(t)\in[0,\sigma_{u_{-1}(\psi)}(T)]$ for
$t\in[0,T]$.

Let $t\in[0,T]$ be such that $\gamma_{u_{-1}(\psi)}(\cdot)$ is
continuous at $\sigma_{u_{-1}(\psi)}(t)$. In this case we certainly
have $\gamma_{u_{-1}(\psi)}(\sigma_{u_{-1}(\psi)}(t))=t$.  Then, under
the assumptions on the functions $u$ and $v$, we easily see that
\begin{eqnarray}
 |\psi(t)|&\leq& |\psi(t)-\psi(0)|+|\psi(0)|\nonumber\\
 &=& |\int_{0}^{\sigma_{u_{-1}(\psi)}(t)}\frac{d \psi(\gamma_{u_{-1}(\psi)}(s))}{ds}ds|+|\psi(0)|\nonumber\\
&\leq& \sqrt{\sigma_{u_{-1}(\psi)}(t)2S_{0T}^{Y}(\psi)}+|\psi(0)|\nonumber\\
&\leq& \sqrt{C_{0}T}\sqrt{2s}+|\psi(0)|\label{Eq:Compactness}
\end{eqnarray}
and similarly if $t,t+h\in[0,T]$ are such that
$\gamma_{u_{-1}(\psi)}(\sigma_{u_{-1}(\psi)}(t))=t$ and
$\gamma_{u_{-1}(\psi)}(\sigma_{u_{-1}(\psi)}(t+h))=t+h$, then
\begin{eqnarray}
 |\psi(t+h)-\psi(t)|&\leq & \sqrt{2s}\sqrt{\sigma_{u_{-1}(\psi)}(t+h)-\sigma_{u_{-1}(\psi)}(t)}\nonumber\\
&\leq& \sqrt{2s}\sqrt{C_{0}h}.\nonumber
\end{eqnarray}
Let $t\in[0,T]$ be such that $\gamma_{u_{-1}(\psi)}(\cdot)$ is not
continuous at $\sigma_{u_{-1}(\psi)}(t)$.
Since for any $t$ we have $\gamma^{-}_{u_{-1}(\psi)}(\sigma_{u_{-1}(\psi)}(t))\leq t \leq \gamma_{u_{-1}(\psi)}(\sigma_{u_{-1}(\psi)}(t))$, Lemma \ref{LemmaTimeChange1a} implies that $\psi(t)=\psi(\gamma_{u_{-1}(\psi)}(\sigma_{u_{-1}(\psi)}(t)))$. Therefore, we have that the calculations in (\ref{Eq:Compactness}) remain valid in this case as well.
This implies part a). For the equicontinuity part b) we can proceed in a similar way and prove that
\begin{displaymath}
 |\psi(t+h)-\psi(t)|\leq \sqrt{2s}\sqrt{C_{0}h}.
\end{displaymath}
This concludes the proof of the theorem.
\end{proof}

\begin{thm}
Let $u$ and $v$ be two strictly increasing functions as in Theorem \ref{MainTheorem} and let $S_{0T}^{Y}(\psi)$  be defined by (\ref{ActionFunctional2}). Then
\begin{enumerate}
 \item For any continuous $\psi:[0,T]\rightarrow \mathbb{R}$ and any $\delta,\eta>0$ there exists an $\epsilon_{0}>0$ such that
  \begin{equation}
    \mathbb{P}_{x}(\sup_{0\leq t \leq T}|Y^{\epsilon}_{t}-\psi(t)|<\delta)\geq \exp \lbrace -\frac{1}{\epsilon}(S_{0T}^{Y}(\psi)+\eta) \rbrace\label{ActionFunctionalStatement1}
  \end{equation}
 for $0<\epsilon<\epsilon_{0}$.
 \item Let $s\in (0,\infty)$ and $\Psi_{s}=\lbrace \psi \in \mathcal{C}([0,T];\mathbb{R}): S_{0T}^{Y}(\psi)\leq s \rbrace $. For any $\delta,h>0$ there exists an  $\epsilon_{0}>0$ such that
      \begin{equation}
  \mathbb{P}_{x}(\rho_{0T} (Y^{\epsilon}_{\cdot}, \Psi_{s})>\delta)\leq \exp \lbrace -\frac{1}{\epsilon}(s-\eta) \rbrace\label{ActionFunctionalStatement2}
 \end{equation}
  for $0<\epsilon<\epsilon_{0}$. Here, $\rho_{0T}(\cdot,\cdot)$ is the uniform metric in $\mathcal{C}([0,T];\mathbb{R})$.

\end{enumerate} \label{MainTheoremHelp1}
\begin{flushright}
$\square$
\end{flushright}
\end{thm}
\begin{proof} Both statements are trivially true if $\psi$ is such
that $S_{0T}^{Y}(\psi)=\infty$. So, we assume that $\psi$ is such
that $S_{0T}^{Y}(\psi)<\infty$.

Throughout the proof of this Theorem we work with a sequence of functions $\left(v_{n}(x)\right)_{n\in N}$ as in the statement of Lemma \ref{LemmaTimeChange1}. Lemma \ref{LemmaTimeChange1} guarantees  that for $\psi$ such that $\Lambda(E_{u_{-1}(\psi)})=0$ relation (\ref{ChangeTimeWithoutN_11}) holds with $\phi=u_{-1}(\psi)$. If the function $\psi$ is such that  $\Lambda(E_{u_{-1}(\psi)})>0$, then we consider a sequence
 $\left(v_{n}(x)\right)_{n\in N}$ such that, in addition to the previous requirements, relation (\ref{ChangeTimeWithoutN_11}) still holds (with the interpretation of $\sigma_{u_{-1}(\psi)}(t)$ given in the statement of Theorem \ref{MainTheorem}). We claim that this restriction can be done without loss of generality. We leave the proof of this claim for the end and we continue with the proof of the Theorem.

(i). Let $n,N>1$ be positive integers that will be chosen appropriately later on and recall the definition of the sequences  $\left(\tau^{n}\right)_{n\in\mathbb{N}}$ and $\left(\sigma^{n}\right)_{n\in\mathbb{N}}$ by (\ref{RandomChangeTimeWithN_1}) and
 (\ref{ChangeTimeWithN_1}) respectively. We have
\begin{eqnarray}
& &\mathbb{P}_{x}(\sup_{0\leq t \leq T}|Y^{\epsilon}_{t}-\psi(t)|<\delta)\geq\nonumber\\
&\geq& \mathbb{P}_{x}(\sup_{0\leq t \leq T}|Y^{\epsilon}_{t}-\psi(t)|<\delta/N,\label{ActionFunctionalStatement1Step1}\\
& &\hspace{0.5cm} \sup_{0\leq t \leq T}|\sqrt{\epsilon}W\left(\sigma_{u_{-1}(\psi)}(t)+[\tau^{n}_{u_{-1}(\sqrt{\epsilon}W)}(t)-\sigma^{n}_{u_{-1}(\psi)}(t)]+\right.\nonumber\\
& &\hspace{0.5cm}\left. +[\tau_{u_{-1}(\sqrt{\epsilon}W)}(t)-\tau^{n}_{u_{-1}(\sqrt{\epsilon}W)}(t)]+
[\sigma^{n}_{u_{-1}(\psi)}(t)-\sigma_{u_{-1}(\psi)}(t)]\right)-\psi(t)|<\delta)\nonumber
\end{eqnarray}
Note that the notation $W_{t}$ and $W(t)$ are used equivalently.

Now by statement (i) of Theorem \ref{VolkonskiiTheorem1} we know that for every $\delta>0$ and $\epsilon>0$ and for $n$ large enough, the following statement holds
\begin{equation}
\mathbb{P}_{x}(\sup_{0\leq t \leq T}|\tau^{n}_{u_{-1}(\sqrt{\epsilon}W)}(t)-\tau_{u_{-1}(\sqrt{\epsilon}W)}(t)|>\frac{\delta}{4})\leq \exp \lbrace -\frac{2}{\epsilon}S_{0T}^{Y}(\psi) \rbrace\label{ActionFunctionalStatement1Step2}
\end{equation}
Moreover, the continuity of the function $\frac{dv_{n}}{du}$ and the fact that $$\tau^{n}_{u_{-1}(\sqrt{\epsilon}W)}(t)=\sigma^{n}_{u_{-1}(Y^{\epsilon})}(t)$$ imply that for any $\delta_{1}>0$
\begin{equation}
\sup_{0\leq t \leq T}|\tau^{n}_{u_{-1}(\sqrt{\epsilon}W)}(t)-\sigma^{n}_{u_{-1}(\psi)}(t)|<\delta_{1}/2\label{ActionFunctionalStatement1Step3}
\end{equation}
for trajectories $Y^{\epsilon}_{t}, 0\leq t\leq T$, such that  $\sup_{0\leq t \leq T}|Y^{\epsilon}_{t}-\psi(t)|<\delta/N$ with a large enough $N$ that is independent of $n$.

By  the choice of the approximating sequence $\left(v_{n}(x)\right)_{n\in N}$ and Lemma \ref{LemmaTimeChange1} we also have that
\begin{equation}
\sup_{0\leq t \leq T}|\sigma^{n}_{u_{-1}(\psi)}(t)-\sigma_{u_{-1}(\psi)}(t)|<\delta_{1}/2\label{ActionFunctionalStatement1Step4}
\end{equation}
for $n$ large enough.

Furthermore, for a one dimensional Wiener process $W_{t}$ we have
\begin{eqnarray}
& &\mathbb{P}_{x}\left(\sqrt{\epsilon}\max_{0\leq t\leq T}\max_{|s|\leq \delta_{1},t+s\geq 0} |W_{t+s}-W_{s}|>\frac{\delta}{4}\right)\nonumber\\
& & \hspace{1cm}\leq \sum_{k=1}^{\left[\frac{T}{\delta_{1}}\right]+1}\mathbb{P}_{x}\left(\sqrt{\epsilon} \max_{0\leq s\leq 2\delta_{1}}\left|W_{\frac{kT}{\delta_{1}}+s}-W_{\frac{kT}{\delta_{1}}}\right|>\frac{\delta}{4}\right)\nonumber\\
& & \hspace{1cm}\leq \left(\frac{T}{\delta_{1}}+1\right)\mathbb{P}_{x}\left( \sqrt{\epsilon}\max_{0\leq s\leq 2\delta_{1}}\left|W_{s}\right|>\frac{\delta}{4}\right)\nonumber\\
& & \hspace{1cm}\leq \frac{T+1}{\delta_{1}} \exp \lbrace-\frac{\delta^{2}}{4\epsilon\delta_{1}}\rbrace\nonumber\\
& & \hspace{1cm}\leq \exp \lbrace -\frac{2}{\epsilon}S_{0T}^{Y}(\psi) \rbrace\label{ActionFunctionalStatement1Step5}
\end{eqnarray}
for $\delta_{1}=\delta/10 S_{0T}^{Y}(\psi)$ and $\epsilon>0$ small enough.

Combining now relations (\ref{ActionFunctionalStatement1Step1})-(\ref{ActionFunctionalStatement1Step5}) and Lemma \ref{LemmaTimeChange1a} we get
\begin{eqnarray}
& &\mathbb{P}_{x}(\sup_{0\leq t \leq T}|Y^{\epsilon}_{t}-\psi(t)|<\delta)\geq\nonumber\\
&\geq& \mathbb{P}_{x}( \sup_{0\leq t \leq T}|\sqrt{\epsilon}W(\sigma_{u_{-1}(\psi)}(t))-\psi(t)|<\frac{\delta}{4})-3\exp \lbrace -\frac{2}{\epsilon}S_{0T}^{Y}(\psi) \rbrace\nonumber\\
&=&\mathbb{P}_{x}( \sup_{0\leq t \leq \sigma_{u_{-1}(\psi)}(T)}|\sqrt{\epsilon}W(t)-\psi(\gamma_{u_{-1}(\psi)}(t))|<\frac{\delta}{4})-3\exp \lbrace -\frac{2}{\epsilon}S_{0T}^{Y}(\psi) \rbrace\nonumber\\
&\geq& C_{0}\exp \lbrace -\frac{1}{\epsilon}(S_{0T}^{Y}(\psi)+\eta) \rbrace
\end{eqnarray}
for $\epsilon$ small enough. In the last inequality we used the well known formula for the action functional of the Gaussian process $\sqrt{\epsilon}W(t)$ on the function $\psi(\gamma_{u_{-1}(\psi)}(t))$ for $0\leq t\leq \sigma_{u_{-1}(\psi)}(T)$.

\vspace{0.2cm}

(ii). By Lemma \ref{ExponentiallyTight} we know that $Y^{\epsilon}_{t}$ is exponential tight. Hence for $\alpha=2s+1$ we have
\begin{displaymath}
\mathbb{P}(\rho_{0T}(Y^{\epsilon}_{\cdot}, K_{2s+1})\geq \delta)<\exp \lbrace -\frac{2s+1}{\epsilon}\rbrace
\end{displaymath}
We have
\begin{eqnarray}
& &\mathbb{P}_{x}(\rho_{0T} (Y^{\epsilon}_{\cdot},\Psi_{s})>\delta)=\mathbb{P}_{x}(\rho_{0T} (Y^{\epsilon}_{\cdot}, \Psi_{s})>\delta, \rho_{0T}(Y^{\epsilon}_{\cdot}, K_{2s+1})< \delta)+\nonumber\\
& &\hspace{3.2cm}+\mathbb{P}_{x}(\rho_{0T} (Y^{\epsilon}_{\cdot}, \Psi_{s})>\delta, \rho_{0T}(Y^{\epsilon}_{\cdot}, K_{2s+1})> \delta) \leq\nonumber\\
&\leq & \mathbb{P}_{x}(\rho_{0T} (Y^{\epsilon}_{s},K_{2s+1}\setminus \Psi_{s})<\delta)+\exp \lbrace -\frac{2s+1}{\epsilon}\rbrace\label{ActionFunctionalStatement2Step1}
\end{eqnarray}
Let now $\psi\in K_{2s+1}\setminus \Psi_{s}$.  Recall that $Y^{\epsilon}_{t}=\sqrt{\epsilon}W\left(\tau_{u_{-1}(\sqrt{\epsilon}W)}(t)\right)$. Hence,  we have
\begin{eqnarray}
& &\mathbb{P}_{x}(\sup_{0\leq t \leq T}|Y^{\epsilon}_{t}-\psi(t)|<2 \delta)=\nonumber\\
&=& \mathbb{P}_{x}(\sup_{0\leq t \leq T}|Y^{\epsilon}_{t}-\psi(t)|<2 \delta,\label{ActionFunctionalStatement2Step2}\\
& &\hspace{0.5cm} \sup_{0\leq t \leq T}|\sqrt{\epsilon}W\left(\sigma_{u_{-1}(\psi)}(t)+[\tau^{n}_{u_{-1}(\sqrt{\epsilon}W)}(t)-\sigma^{n}_{u_{-1}(\psi)}(t)]+\right.\nonumber\\
& &\hspace{0.5cm}\left. +[\tau_{u_{-1}(\sqrt{\epsilon}W)}(t)-\tau^{n}_{u_{-1}(\sqrt{\epsilon}W)}(t)]+
[\sigma^{n}_{u_{-1}(\psi)}(t)-\sigma_{u_{-1}(\psi)}(t)]\right)-\psi(t)|<2 \delta)\nonumber
\end{eqnarray}
Using (\ref{ActionFunctionalStatement1Step2})-(\ref{ActionFunctionalStatement1Step5}) and Lemma \ref{LemmaTimeChange1a}, the latter implies that for $n$ large enough and $\delta$ small enough we have
\begin{eqnarray}
& &\mathbb{P}_{x}(\sup_{0\leq t \leq T}|Y^{\epsilon}_{t}-\psi(t)|<2 \delta)\leq\nonumber\\
&\leq& \mathbb{P}_{x}( \sup_{0\leq t \leq T}|\sqrt{\epsilon}W(\sigma_{u_{-1}(\psi)}(t))-\psi(t)|<8\delta)+3\exp \lbrace -\frac{2s+1}{\epsilon} \rbrace\nonumber\\
&=&\mathbb{P}_{x}( \sup_{0\leq t \leq \sigma_{u_{-1}(\psi)}(T)}|\sqrt{\epsilon}W(t)-\psi(\gamma_{u_{-1}(\psi)}(t))|<8\delta)+3\exp \lbrace -\frac{2s+1}{\epsilon} \rbrace\nonumber\\
&\leq& \exp \lbrace -\frac{1}{\epsilon}(S_{0T}^{Y}(\psi)-\eta) \rbrace +3\exp \lbrace -\frac{2s+1}{\epsilon} \rbrace\nonumber\\
&\leq& C_{0}\exp \lbrace -\frac{1}{\epsilon}(s-\eta) \rbrace  \rbrace
\label{ActionFunctionalStatement2Step3}
\end{eqnarray}
for $\epsilon$ small enough. In the last inequality we used the well known formula for the action functional of the Gaussian process $\sqrt{\epsilon}W(t)$ on the function $\psi(\gamma_{u_{-1}(\psi)}(t))$ for $0\leq t\leq \sigma_{u_{-1}(\psi)}(T)$ and that for $\psi\in K_{2s+1}\setminus \Psi_{s}$ we have $S_{0T}^{Y}(\psi)\geq s$.

Let now $\psi^{i}$ for $i\in\lbrace 1,\cdots, N\rbrace$ be a finite $\delta$-net of $K_{2s+1}\setminus \Psi_{s}$. Then (\ref{ActionFunctionalStatement2Step1}) and (\ref{ActionFunctionalStatement2Step3}) imply that
\begin{equation}
\mathbb{P}_{x}(\rho_{0T}(Y^{\epsilon}_{s},\Psi_{s})>\delta)\leq C_{0}\exp \lbrace -\frac{1}{\epsilon}(s-\eta) \rbrace  \rbrace \label{ActionFunctionalStatement2Step4}
\end{equation}
for $\epsilon$ small enough. This concludes the proof of part (ii) of the Theorem.

It remains to prove the claim made in the beginning of the proof. Let us pick a point $z\in \left(U \cup V\right)\setminus V_{d}$ and let us write for notational convenience $\phi=u_{-1}(\psi)$. Denote $E^{z}_{\phi}=\{t\in[0,T]: \phi_{t}=z\}$. Essentially, we have two cases
\begin{enumerate}
\item Assume that $E^{z}_{\phi}$ is an interval, for example $E^{z}_{\phi}=[t_{0},t_{1}]\subset [0,T]$. We will have that $\dot{\phi}_{t}=0$ for every $t\in(t_{0},t_{1})$. Then, it is easy to see that $S_{E^{z}_{\phi}}(\phi)=S_{t_{0}t_{1}}(\phi)=0$ (e.g., from expression (\ref{ActionFunctionalreduced})).
\item Assume that $E^{z}_{\phi}$ is not an interval. Then, one can use Theorem A.6.3 in \cite{DupuisEllis} or Problem $11$ on pages $334-335$ of \cite{EithierKurtz} to claim that the Lebesgue measure of the set $\{t\in[0,T]: \phi_{t}=z, \dot{\phi}_{t}\neq 0\}$ is zero (due to absolute continuity).
\end{enumerate}
Hence, in either case we have that $S_{E^{z}_{\phi}}(\phi)=0$. In other words, even though, for the case $\Lambda (E_{\phi})>0$, the limit of $\sigma^{n}_{\phi}(t)$ as $n\rightarrow\infty$ is affected by the approximating sequence $\left(v_{n}(x)\right)_{n\in N}$, the corresponding action functional is not. Thus, we can make the convention that was made in part (i) of Theorem \ref{MainTheorem}.

\end{proof}
We conclude this section with the proof of Theorem \ref{MainTheorem}.
\begin{proof}[Proof of Theorem \ref{MainTheorem}]
Lemma
\ref{LemmaTimeChange1} is essentially statement (i) of Theorem
\ref{MainTheorem}.

As far as statement (ii) of Theorem
\ref{MainTheorem} is concerned, we have the following.
By Remark \ref{R:DiscontinuityPoints} and Theorems \ref{MainTheoremHelp2} and \ref{MainTheoremHelp1} we have that
$\frac{1}{\epsilon}S_{0T}^{Y}(\psi)$ is the action functional for the process $\left(Y^{\epsilon}_{t}\right)_{t\in[0,T]}$ on $\mathcal{C}([0,T];\mathbb{R})$ as $\epsilon\downarrow 0$. Then by the contraction principle we have that the action functional for the process $\left(X^{\epsilon}_{t}\right)_{t\in[0,T]}$ on $\mathcal{C}([0,T];\mathbb{R})$ as $\epsilon\downarrow 0$ is given by $\frac{1}{\epsilon}S_{0T}(\phi)$, where
\begin{eqnarray}
 S_{0T}(\phi)&=&\inf \lbrace S_{0T}^{Y}(\psi): \psi=u(\phi)\rbrace\nonumber\\
&=&S_{0T}^{Y}(u(\phi))\nonumber
\end{eqnarray}
The compactness of the set $\Phi_{s}=\lbrace \phi \in
\mathcal{C}([0,T];\mathbb{R}): S_{0T}(\phi)\leq s \rbrace $ and the lower
semicontinuity of $S_{0T}(\phi)$ follows immediately from the
corresponding statements for $\Psi_{s}$ and $S_{0T}^{Y}(\psi)$.
\end{proof}

\section{Proof of auxiliary results}
In this section we prove Lemma \ref{LemmaTimeChange1},  Lemma \ref{LemmaTimeChange1a}, Lemma \ref{LemmaTimeChange2}, Proposition \ref{Proposition21} and Lemma \ref{ExponentiallyTight}.

\begin{proof}[Proof of Lemma \ref{LemmaTimeChange1}]
A lemma similar to this one is stated without proof in \cite{Volkonskii2}. Here, we provide for completeness a sketch of the proof for our case of interest.

Let $\phi:[0,T]\rightarrow \mathbb{R}$ be a continuous function in $[0,T]$, i.e. $\phi \in \mathcal{C}([0,T];\mathbb{R})$.
Recall that the functions $\sigma^{n}_{\phi}(t)$ are defined by the formula
\begin{displaymath}
\sigma^{n}_{\phi}(t)=\int_{0}^{t}[\frac{1}{2}\frac{dv_{n}}{du}(\phi_{s})]^{-1}ds.
\end{displaymath}

It is easy to see now, that it is enough to prove that
$\lim_{n\rightarrow \infty}\sigma^{n}_{\phi}(t)$ exists for any
$t\in[0,T]$ independently of the choice of the sequence $\left(v_{n}\right)_{n\in N}$.
Then, uniformity follows from the latter and the fact that the first
derivatives of the functions $\sigma^{n}_{\phi}(t)$ are bounded
uniformly in $n$ and $t\in[0,T]$. The assumptions on the functions
$u$ and $v_{n}$ guarantee the boundedness of the first derivatives
of $\sigma^{n}_{\phi}(t)$.

It is clear that $\lim_{n\rightarrow \infty}\sigma^{n}_{\phi}(t)$
exists, independently of the choice of the sequence $\left(v_{n}\right)_{n\in N}$, if the
Lebesgue measure of $V_{d,\phi}$ is zero, i.e.  $\Lambda (
V_{d,\phi})=0$. In this case, the $\lim_{n\rightarrow
\infty}\sigma^{n}_{\phi}(t)$ is continuous and strictly increasing
function of $t$.

Hence, it remains to consider the case $\Lambda (
V_{d,\phi})>0$. It is enough to prove that for any $\epsilon>0$ there is a $n_{0}(\epsilon)>0$ such that
\begin{displaymath}
|\int_{V_{d,\phi}}[\frac{1}{2}\frac{dv_{n}}{du}(\phi_{s})]^{-1}-[\frac{1}{2}\frac{dv_{m}}{du}(\phi_{s})]^{-1}ds|<\epsilon \textrm{ for any } n,m\geq n_{0}(\epsilon).
\end{displaymath}
We write
\begin{eqnarray}
& &|\int_{V_{d,\phi}}[\frac{1}{2}\frac{dv_{n}}{du}(\phi_{s})]^{-1}-[\frac{1}{2}\frac{dv_{m}}{du}(\phi_{s})]^{-1}ds|
\leq\nonumber\\
&\leq& |\int_{V_{d,\phi}\setminus U_{\phi}}[\frac{1}{2}\frac{dv_{n}}{du}(\phi_{s})]^{-1}-[\frac{1}{2}\frac{dv_{m}}{du}(\phi_{s})]^{-1}ds|+\nonumber\\
&+&|\int_{V_{d,\phi}\bigcap U_{\phi}}[\frac{1}{2}\frac{dv_{n}}{du}(\phi_{s})]^{-1}-[\frac{1}{2}\frac{dv_{m}}{du}(\phi_{s})]^{-1}ds|\nonumber
\end{eqnarray}
If $\Lambda (V_{d,\phi}\bigcap  U_{\phi})=0$, then the second term in the inequality above is zero and it is easily seen that the first term can be made arbitrarily small for $n,m$ large enough.

If, on the other hand, $\Lambda ( U_{\phi}\bigcap V_{d,\phi})>0$,
then we may define
$$\lim_{n\rightarrow \infty}
\int_{U_{\phi}\bigcap
V_{d,\phi}}[\frac{1}{2}\frac{dv_{n}}{du}(\phi_{s})]^{-1}ds=0$$
and the result follows. Therefore, in the case $\Lambda(E_{\phi})=0$, the  $\lim_{n\rightarrow \infty}\sigma^{n}_{\phi}(t)$ exists and
the limit is independent of the approximating sequence $\left(v_{n}\right)_{n\in N}$. Finally, it is easily seen that the limit is non decreasing and continuous in $t$.
\end{proof}

\vspace{0.2cm}

\begin{proof}[Proof of Lemma \ref{LemmaTimeChange1a}]
It is clear that $\gamma_{\phi}(t)$ is right continuous. Moreover, it is easy to see that
continuity of $\sigma_{\phi}(\cdot)$ implies that $\sigma_{\phi}(\gamma_{\phi}^{-}(t))=\sigma_{\phi}(\gamma_{\phi}(t))$. This implies that $\phi(s)\in V_{d}$ almost everywhere in $s\in[\gamma_{\phi}^{-}(t), \gamma_{\phi}(t)]$. Recall that $V_{d}$ is the set of discontinuity points for function $v(x)$.

Let now $x_{0}\in V_{d}$ such that $\phi(\gamma_{\phi}^{-}(t))=x_{0}$. Define
\begin{displaymath}
s_{o}=\sup\lbrace s:s\in[\gamma_{\phi}^{-}(t), \gamma_{\phi}(t)], \phi(s)=x_{0},
 \phi(\rho)\notin V_{d}\setminus \lbrace x_{0}\rbrace \textrm{ for all } \rho< s\rbrace
\end{displaymath}
If $s_{o}=\gamma_{\phi}(t)$ then $\phi(s)$ is constant almost everywhere in $s\in[\gamma_{\phi}^{-}(t), \gamma_{\phi}(t)]$. Therefore,  $\phi(s)$ is constant everywhere in $s\in[\gamma_{\phi}^{-}(t), \gamma_{\phi}(t)]$ since $\phi(s)$ is continuous.

Assume that there is some $x_{1}\in V_{d}$ with $x_{1}\neq x_{0}$ such that $\phi(s)=x_{1}$ for some $s\in[\gamma_{\phi}^{-}(t), \gamma_{\phi}(t)]$. In particular, define
\begin{displaymath}
s_{1}=\inf\lbrace s:s\in(s_{0}, \gamma_{\phi}(t)], \phi(s)\in V_{d}\setminus\lbrace x_{0}\rbrace\rbrace.
\end{displaymath}
We write $\phi(s_{1})=x_{1}$. Of course, if $s_{0}=s_{1}$ then we have a contradiction since $\phi(s_{0})=x_{0}$ and $\phi(s_{1})=x_{1}$. So, we assume that $s_{0}<s_{1}$. In this case we clearly have that $\sigma_{\phi}(s_{0})<\sigma_{\phi}(s_{1})$. However, since $[s_{0},s_{1}]\subset [\gamma_{\phi}^{-}(t), \gamma_{\phi}(t)]$ and $\sigma_{\phi}(\cdot)$ is non decreasing and continuous, the latter clearly contradicts $\sigma_{\phi}(\gamma_{\phi}^{-}(t))=\sigma_{\phi}(\gamma_{\phi}(t))$. Hence, such an $x_{1}$ does not exist. The latter implies that $\phi(s)$ is constant almost everywhere in $s\in[\gamma_{\phi}^{-}(t), \gamma_{\phi}(t)]$. Therefore,  $\phi(s)$ is constant everywhere in $s\in[\gamma_{\phi}^{-}(t), \gamma_{\phi}(t)]$ since $\phi(s)$ is continuous.

\end{proof}

\vspace{0.2cm}

\begin{proof}[Proof of Lemma \ref{LemmaTimeChange2}]
Let $\phi^{n}$ be a sequence of functions in $\mathcal{C}([0,T];\mathbb{R})$ that converges to $\phi$ uniformly in $\mathcal{C}([0,T];\mathbb{R})$. We only prove parts (ii) and (iii). Part (i) is easily seen to hold by the uniform convergence of $\phi^{n}$ to $\phi$.

Let $t_{*}\in [0,\sigma_{\phi}(T)]$ be a continuity point of $\gamma_{\phi}(t)$. Of course, $\gamma_{\phi}(t)$ can only have countable many points of discontinuity.

Let $s_{*}\in[0,T]$ be such that $t_{*}=\sigma_{\phi}(s_{*})$. Such an $s_{*}$ exists because $\sigma_{\phi}(s)$ is continuous. By part (i) we have that for any $\epsilon>0$ there is an $n_{0}(\epsilon)\in\mathbb{N}$ such that
\begin{displaymath}
|\sigma_{\phi^{n}}(s)-\sigma_{\phi}(s)|<\epsilon
\end{displaymath}
for every $s\in[0,T]$ and $n\geq n_{0}(\epsilon)$.

The latter and the fact that $\gamma_{\phi^{n}}(t)$ is non-decreasing give us
\begin{displaymath}
\gamma_{\phi^{n}}(\sigma_{\phi^{n}}(s_{*})-\epsilon)\leq \gamma_{\phi^{n}}(\sigma_{\phi}(s_{*})) \leq \gamma_{\phi^{n}}(\sigma_{\phi^{n}}(s_{*})+\epsilon)
\end{displaymath}
For $n\geq n_{0}(\epsilon)$ we have
\begin{eqnarray}
\gamma_{\phi^{n}}(\sigma_{\phi^{n}}(s_{*})+\epsilon)&=&\inf\lbrace s:\sigma_{\phi^{n}}(s) > \sigma_{\phi^{n}}(s_{*})+\epsilon\rbrace\nonumber\\
&\leq& \inf\lbrace s:\sigma_{\phi}(s) > \sigma_{\phi^{n}}(s_{*})+2\epsilon\rbrace\nonumber\\
&\leq& \inf\lbrace s:\sigma_{\phi}(s) > \sigma_{\phi}(s_{*})+3\epsilon\rbrace\nonumber\\
&=&\gamma_{\phi}(\sigma_{\phi}(s_{*})+3\epsilon)\nonumber
\end{eqnarray}
Likewise, for $n$ large enough
\begin{displaymath}
\gamma_{\phi^{n}}(\sigma_{\phi^{n}}(s_{*})-\epsilon)\geq  \gamma_{\phi}(\sigma_{\phi}(s_{*})-3\epsilon)
\end{displaymath}
Therefore, for $n$ large enough, we have
\begin{equation}
\gamma_{\phi}(\sigma_{\phi}(s_{*})-3\epsilon)\leq \gamma_{\phi^{n}}(\sigma_{\phi}(s_{*})) \leq \gamma_{\phi}(\sigma_{\phi}(s_{*})+3\epsilon)\label{LemmaTimeChange2Step1}
\end{equation}
Therefore, (\ref{LemmaTimeChange2Step1}) implies that
\begin{displaymath}
 \gamma_{\phi^{n}}(\sigma_{\phi}(s_{*}))\rightarrow \gamma_{\phi}(\sigma_{\phi}(s_{*})) \textrm{ as } n\rightarrow\infty,
\end{displaymath}
or in other words
\begin{displaymath}
 \gamma_{\phi^{n}}(t_{*})\rightarrow \gamma_{\phi}(t_{*}) \textrm{ as } n\rightarrow\infty,
\end{displaymath}
which concludes the proof of part (ii) of the lemma.

Lastly, we prove part (iii) of the lemma. Let $t\in[0,\sigma_{\phi}(T)]$. We write
\begin{eqnarray}
 |\phi^{n}(\gamma_{\phi^{n}}(t))-\phi(\gamma_{\phi}(t))|&\leq&
 |\phi^{n}(\gamma_{\phi}(t))-\phi(\gamma_{\phi}(t))|\label{LemmaTimePart3Step1}\\
&+& |\phi^{n}(\gamma_{\phi^{n}}(t))-\phi^{n}(\gamma_{\phi}(t))|\nonumber
\end{eqnarray}
The uniform convergence of $\phi^{n}$ to $\phi$ guarantees that the first term in the right hand side of (\ref{LemmaTimePart3Step1}) can be made arbitrarily small for $n$ large enough. Moreover, part (ii), guarantees that the second term can be arbitrarily small provided that $t$ is a continuity point of $\gamma_{\phi}(\cdot)$. Hence, it is enough to consider the case where $t$ is not a continuity point of $\gamma_{\phi}(\cdot)$. We claim that the following two statements hold.
\begin{enumerate}[a)]
 \item For every $\epsilon>0$ there is a $n_{0}(\epsilon)>0$ such that for every $t\in[0,\sigma_{\phi}(T)]$ and for every $n>n_{0}(\epsilon)$ we have that
\begin{displaymath}
 \gamma_{\phi^{n}}(t)\in[\gamma_{\phi}(t-\epsilon), \gamma_{\phi}(t+\epsilon)].
\end{displaymath}
\item The function $\phi(s)$ is constant for $s\in[\gamma_{\phi}^{-}(t), \gamma_{\phi}(t)]$, where we set $\gamma_{\phi}^{-}(t)=\lim_{s\rightarrow t^{-}}\gamma_{\phi}(s)$.
\end{enumerate}
These statements together with the uniform convergence $\phi^{n}$ to $\phi$ guarantee that the second term in the right hand side of (\ref{LemmaTimePart3Step1}) can be made arbitrarily small for $n$ large enough even if $t$ is not a continuity point of $\gamma_{\phi}(\cdot)$. Hence, it remains to prove the claim. Part a) follows by an arguement similar to the one that was used in the proof of part (ii) of this lemma (see (\ref{LemmaTimeChange2Step1})) and part b) is Lemma \ref{LemmaTimeChange1a}.

This concludes the proof of the lemma.
\end{proof}

\begin{proof}[Proof of Proposition \ref{Proposition21}]
Recall that $X^{\epsilon}_{t}=u_{-1}[\sqrt{\epsilon}W_{\tau_{u_{-1}(\sqrt{\epsilon}W)}(t)}]$, where $\tau_{u_{-1}(\sqrt{\epsilon}W)}(t)$ is defined as in (\ref{RandomChangeTimeWithoutN_1}) with $\sqrt{\epsilon}W$ in place of $W$. Let us also define $\hat{X}_{t}=u_{-1}[W_{\hat{\tau}_{u_{-1}(W)}(t)}]$, where $\hat{\tau}_{u_{-1}(W)}(t)$ is defined similarly to (\ref{RandomChangeTimeWithoutN_1}).
Then, we easily see that
\begin{eqnarray}
t&=&\int^{\tau^{n}_{u_{-1}(\sqrt{\epsilon}W)}(t)}_{0}\frac{1}{2}\frac{dv_{n}}{du}(u_{-1}(\sqrt{\epsilon}W_{s}))ds\nonumber\\
&=&\frac{1}{\epsilon}\int^{\epsilon\tau^{n}_{u_{-1}(\sqrt{\epsilon}W)}(t)}_{0}\frac{1}{2}\frac{dv_{n}}{du}(u_{-1}(W_{s}))ds.\nonumber
\end{eqnarray}
On the other hand, it is also true that
\begin{displaymath}
t=\frac{1}{\epsilon}\int^{\hat{\tau}^{n}_{u_{-1}(W)}(\epsilon t)}_{0}\frac{1}{2}\frac{dv_{n}}{du}(u_{-1}(W_{s}))ds.
\end{displaymath}
The latter imply that
\begin{displaymath}
\int^{\epsilon\tau^{n}_{u_{-1}(\sqrt{\epsilon}W)}(t)}_{0}\frac{1}{2}\frac{dv_{n}}{du}(u_{-1}(W_{s}))ds=
\int^{\hat{\tau}^{n}_{u_{-1}(W)}(\epsilon
t)}_{0}\frac{1}{2}\frac{dv_{n}}{du}(u_{-1}(W_{s}))ds
\end{displaymath}
Taking into account that $\tau^{n}_{u_{-1}(\sqrt{\epsilon}W)}(t)$
and $\hat{\tau}^{n}_{u_{-1}(W)}(t)$ are strictly increasing in $t$
and that $\frac{dv_{n}}{du}$ is strictly positive, we get that
almost surely
\begin{displaymath}
\epsilon\tau^{n}_{u_{-1}(\sqrt{\epsilon}W)}(t)=\hat{\tau}^{n}_{u_{-1}(W)}(\epsilon
t).
\end{displaymath}
The latter implies that
\begin{eqnarray}
X^{\epsilon}_{t}&=&u_{-1}[\sqrt{\epsilon}W_{\tau_{u_{-1}(\sqrt{\epsilon}W)}(t)}]\nonumber\\
&=&u_{-1}[W_{\hat{\tau}_{u_{-1}(W)}(\epsilon t)}]\nonumber\\
&=&\hat{X}_{\epsilon t}.\label{Proposition21Step3}
\end{eqnarray}
Let now $I$ be an interval in $\mathbb{R}$ and $T_{I}$ and $\hat{T}_{I}$
be the exit times for $X^{\epsilon}_{t}$, $\hat{X}_{t}$ from $I$
respectively. Then using (\ref{Proposition21Step3}), the
infinitesimal generator of $X^{\epsilon}_{t}$ is
\begin{eqnarray}
\lim_{dI\rightarrow 0}\frac{\mathbb{E}_{x}f(X^{\epsilon}_{T_{I}})-f(x)}{\mathbb{E}_{x}T_{I}}&=&\lim_{dI\rightarrow 0}\frac{\mathbb{E}_{x}f(\hat{X}_{\hat{T}_{I}})-f(x)}{\mathbb{E}_{x}\hat{T}_{I}}\frac{\mathbb{E}_{x}\hat{T}_{I}}{\mathbb{E}_{x}T_{I}}\nonumber\\
&=&\epsilon D_{v}D_{u},\nonumber
\end{eqnarray}
where $dI$ is the length of $I$. This concludes the proof of the
proposition.
\end{proof}

\vspace{0.2cm}

\begin{proof}[Proof of Lemma \ref{ExponentiallyTight}]
The result can be easily derived by the representation
$Y^{\epsilon}_{t}=\sqrt{\epsilon}W_{\tau_{u_{-1}(\sqrt{\epsilon}W)}(t)}$
and Theorem 4.1 of \cite{FengKurtz}.
\end{proof}

\section{Generalized reaction-diffusion equations and some results on wave front propagation}
In this section we discuss reaction-diffusion equations governed by
a generalized elliptic operator $D_{v}D_{u}$. We will refer to them
as generalized reaction diffusion equations. We apply Theorem
\ref{MainTheorem} to the problem of wave front propagation for these
type of reaction-diffusion equations in the case where the
non-linear term is of K-P-P type.

Let $D_{v}D_{u}$ be the operator introduced in the introduction. For $f\in\mathcal{D}(D_{v}D_{u})$, i.e. for functions that belong to the domain of definition of the $D_{v}D_{u}$ operator, consider the following reaction diffusion equation
\begin{eqnarray}
f_{t}(t,x)&=& D_{v}D_{u}f(t,x)+c(x,f(t,x))f(t,x)\nonumber\\
f(0,x)&=&g(x)\label{RDE1}
\end{eqnarray}
We shall consider the generalized solution to (\ref{RDE1}). We define the operator
\begin{displaymath}
Af=-f_{t}+D_{v}D_{u}f.
\end{displaymath}
As it is well known, there exists a corresponding Markov family $Y_{s}=(t-s,X_{s})$ in the state space $(-\infty, T]\times \mathbb{R}, T>0$. Here $X_{s}$ is the strong Markov process governed by the operator $ D_{v}D_{u}$. Moreover, we define $f(t,x)=g(x)$ for $t\leq 0$. Using
the Feynman-Kac formula, the solution to this problem  may be written as follows:
\begin{equation}
f(t,x)=\mathbb{E}_{x}g(X_{t})e^{\int_{0}^{t}c(X_{s},f(t-s,X_{s}))ds}
\label{FeynmanKacFormula}
\end{equation}
We shall call the solution to equation
(\ref{FeynmanKacFormula}) the generalized solution to equation
(\ref{RDE1}). Throughout this section, we will make the following assumption.
\begin{ass}
The function $c(x,f)$ is uniformly bounded in all arguments, continuous in $x$ and Lipschitz continuous in $f$. 
The initial profile $g(x)$ is a bounded, nonnegative  function that can have at most a finite number of simple discontinuities.
\begin{flushright}
$\square$
\end{flushright}
\end{ass}

One can prove, via the standard method of successive approximations,
that under the aforementioned assumption, there exists a unique
generalized solution for the problem (\ref{RDE1}). Namely, the
equation (\ref{FeynmanKacFormula}) has a unique solution (see
chapter 5 of \cite{F1} for more details).

\vspace{0.2cm}

Generalized reaction diffusion equations, like (\ref{RDE1}), can
appear in applications as, for example, the limit of a family of
standard reaction-diffusion equations.

Let us demonstrate this in a simple case. Consider the family of problems
\begin{eqnarray}
f^{n}_{t}(t,x)&=& L_{n}f^{n}(t,x)+c(x,f^{n}(t,x))f^{n}(t,x)\nonumber\\
f^{n}(0,x)&=&g(x)\label{RDE2}
\end{eqnarray}
where $L_{n}$ is a family of standard second order elliptic operators
\begin{equation}
L_{n}f(x)=\frac{1}{2}a_{n}(x) \frac{d^{2}f(x)}{dx^{2}}+b_{n}(x)\frac{df(x)}{dx}.\label{FamilySmoothOperator}
\end{equation}
Assume that the limits of the coefficients  $a_{n}(x)$ and $b_{n}(x)$ are discontinuous as follows
\begin{eqnarray}
\lim_{n\rightarrow \infty}a_{n}(x)= a(x)=\begin{cases}a_{+}(x), & x>0 \\
                                                     a_{-}(x), & x<0. \end{cases}\nonumber
\end{eqnarray}
and
\begin{eqnarray}
\lim_{n\rightarrow \infty}b_{n}(x)= b(x)=\begin{cases}b_{+}(x), & x>0 \\
                                                     b_{-}(x), & x<0. \end{cases}\nonumber
\end{eqnarray}

where $a(x)$ and $b(x)$ may not be defined or be discontinuous at
$x=0$. Define
\begin{displaymath}
u_{n}(x)=\int_{0}^{x} e^{-\int_{0}^{y}\frac{2b_{n}(z)}{a_{n}(z)}dz} dy \hspace{0.3cm}\textrm{ and }\hspace{0.3cm} v_{n}(x)=\int_{0}^{x}\frac{2}{a_{n}(y)}e^{\int_{0}^{y}\frac{2b_{n}(z)}{a_{n}(z)}dz} dy.
\end{displaymath}
We observe that $D_{v_{n}}D_{u_{n}}f=L_{n}f$. Let $X^{n}_{t}$ be the
one dimensional Markov process with infinitesimal generator $L_{n}$
and let $\tau^{n}(-\delta,\delta)=\inf \lbrace t: X^{n}_{t}\notin
(-\delta,\delta)\rbrace$. Define the quantities
\begin{eqnarray}
P_{r}&=&\lim_{\delta\downarrow 0}\lim_{n\rightarrow
\infty}\mathbb{P}_{x}(X^{n}_{\tau^{n}}=\delta)=\lim_{\delta\downarrow
0}\lim_{n\rightarrow
\infty}\frac{u_{n}(x)-u_{n}(-\delta)}{u_{n}(\delta)-u_{n}(-\delta)}\nonumber\\
P_{l}&=&\lim_{\delta\downarrow 0}\lim_{n\rightarrow
\infty}\mathbb{P}_{x}(X^{n}_{\tau^{n}}=-\delta)=\lim_{\delta\downarrow
0}\lim_{n\rightarrow
\infty}\frac{u_{n}(\delta)-u_{n}(x)}{u_{n}(\delta)-u_{n}(-\delta)}\nonumber\\
\kappa &=&\lim_{\delta\downarrow 0}\lim_{n\rightarrow
\infty}\mathbb{E}_{x}\frac{1}{\delta}\tau^{n}(-\delta,\delta).\nonumber
\end{eqnarray}
The function $m_{n}(x)=\mathbb{E}_{x}\tau^{n}(-\delta,\delta)$ is
solution to the equation $D_{v_{n}}D_{u_{n}}m_{n}(x)=-1$ with
boundary conditions $m_{n}(-\delta)=m_{n}(\delta)=0$.

If $P_{r}=P_{l}=\frac{1}{2}$ and $\kappa=0$, then the limit (in
distribution) of $X^{n}_{t}$ behaves locally like a Wiener process.
But, of course, this is not the case in general. Define the
functions
\begin{eqnarray}
u(x)=\begin{cases}\frac{1}{P_{r}}\int_{0}^{x} e^{-\int_{0}^{y}\frac{2b(z)}{a(z)}dz} dy, & x\geq0 \\
            \frac{1}{P_{l}}\int_{0}^{x} e^{-\int_{0}^{y}\frac{2b(z)}{a(z)}dz} dy   , & x<0. \end{cases}\nonumber
\end{eqnarray}
\begin{eqnarray}
v(x)=\begin{cases}\kappa+P_{r}\int_{0}^{x}\frac{2}{a(y)}e^{\int_{0}^{y}\frac{2b(z)}{a(z)}dz}
dy, & x\geq0 \\
            P_{l} \int_{0}^{x}\frac{2}{a(y)}e^{\int_{0}^{y}\frac{2b(z)}{a(z)}dz}dy   , & x<0. \end{cases}
\nonumber
\end{eqnarray}

and assume that $P_{r}$, $P_{l}$, $\kappa$ and that the limit
$\lim_{n\rightarrow\infty}
e^{-\int_{0}^{y}\frac{2b_{n}(z)}{a_{n}(z)}dz}$ exists for all
$y\in\mathbb{R}$. It is easy to see that
\begin{eqnarray}
u(x)&=&\lim_{n\rightarrow\infty}u_{n}(x) \textrm{ for every } x\in
\mathbb{R}\nonumber\\
v(x)&=&\lim_{n\rightarrow\infty}v_{n}(x) \textrm{ for every } x\in
\mathbb{R}\setminus\{0\}\nonumber
\end{eqnarray}
Then, it can be shown (see \cite{Hyejin} for more details) that
\begin{displaymath}
\lim_{n\rightarrow\infty} f^{n}(t,x)=f(t,x),
\end{displaymath}
where $f_{n}(t,x)$ and $f(t,x)$ are the
generalized solutions to (\ref{RDE2}) and (\ref{RDE1}) respectively.
In this case, the domain of definition of the $D_{v}D_{u}$ operator
is
\begin{eqnarray}
     \mathcal{D} (D_{v}D_{u})=\lbrace &f:& f\in \mathcal{C}_{c}(\mathbb{R})\textrm{, with }f_{x},f_{xx}\in\mathcal{C}(\mathbb{R}\setminus\{0\}),\nonumber\\
& &  P_{r}f'_{+}(0)-P_{l}f'_{-}(0)=\kappa DvDuf(0) \textrm{ and}\nonumber\\
& &  DvDuf(0)=\lim_{x\rightarrow 0^{+}}DvDuf(x)=\lim_{x\rightarrow 0^{-}}DvDuf(x) \rbrace. \nonumber
\end{eqnarray}

Let us study now the problem of wave front propagation for the following equation. For $f\in\mathcal{D}(D_{v}D_{u})$ consider the generalized solution to the following reaction diffusion equation
\begin{eqnarray}
f^{\epsilon}_{t}(t,x)&=&\epsilon D_{v}D_{u}f^{\epsilon}(t,x)+\frac{1}{\epsilon}c(x,f^{\epsilon}(t,x))f^{\epsilon}(t,x)\nonumber\\
f^{\epsilon}(0,x)&=&g(x)\label{RDE3}
\end{eqnarray}

For brevity, we consider the initial profile of (\ref{RDE3}) to be
given by $g(x)=\chi_{x\leq 0}$, where $\chi_{x\leq 0}$ is the
characteristic function of the set $\{x: x\leq 0\}$. Moreover, the
non linear function $c(x,f)$ is assumed to be of
Kolmogorov-Petrovskii-Piskunov (K-P-P) type, i.e. it is Lipschitz
continuous in $f\in\mathbb{R}$, positive for $f<1$, negative for
$f>1$ and $c(x)=c(x,0)=\max_{0\leq f\leq 1}c(x,f)$. Generalized
reaction diffusion equations that have a K-P-P type nonlinear term
are called K-P-P generalized reaction diffusion equations.

It is not difficult to see that the classical results of Freidlin \cite{F1} on wave front propagation of K-P-P reaction diffusion equations hold in this case as well. Let us define
\begin{equation}
W(t,x)=\sup\{\int_{0}^{t}c(\phi_{s})ds-S_{0t}(\phi):\phi \in \mathcal{C}_{0,t}, \phi_{0}=x, \phi_{t}\leq 0\}.\label{DefinitionOfWIntro}
\end{equation}
where $c(x)=c(x,0)=\max_{0\leq f\leq 1}c(x,f)$ and $S_{0t}(\phi)$,
defined by (\ref{ActionFunctional}), is the action functional for
the Markov process $X^{\epsilon}_{t}$ whose infinitesimal generator
is $\epsilon D_{v}D_{u}$.

We say that condition (N) is satisfied if for any $t>0$ and $(t,x)\in
\{(t,x):W(t,x)=0\}:$

\begin{eqnarray}
W(t,x)=\sup\{
\int_{0}^{t}c(\phi_{s})ds-S_{0t}(\phi)&:&\phi_{0}=x,\phi_{t}\leq 0,\nonumber\\
& & (t-s,\phi_{s})\in \{(t,x):W(t,x)<0\}\}.\nonumber
\end{eqnarray}
\begin{thm}(Freidlin \cite{F1}). Let $f^{\epsilon}(t,x)$ be the unique generalized solution to (\ref{RDE3}). Then,
under condition (N) we have:
\begin{eqnarray}
\lim_{\epsilon \downarrow 0}f^{\epsilon}(t,x)=\begin{cases}1, & W(t,x)>0 \\
                                                     0, & W(t,x)<0. \end{cases}
\end{eqnarray}
The convergence is uniform on every compactum lying in the region $\lbrace(t,x):t>0, x\in\mathbb{R}, W(t,x)>0\rbrace$ and $\lbrace(t,x):t>0, x\in\mathbb{R}, W(t,x)<0\rbrace$ respectively.\label{TheoremFreidlin1}
\end{thm}
Hence, the equation $W(t,x)=0$ defines the position of the interface
(wavefront) between areas where $f^{\epsilon}$ (for $\epsilon>0$
small enough) is close to $0$ and to $1$. Moreover, $W(t,x)$ is a
continuous function, increasing in $t$.

We shall consider a simple example that illustrates the
applicability of Theorem \ref{MainTheorem}.
Assume, for brevity, that
\begin{eqnarray}
u(x)&=&x\nonumber\\
v(x)&=&\begin{cases}Ax, & x<x_{1} \\
              \kappa+Ax, & x_{1}\leq x \leq x_{2} \\
              \kappa+Ax+B(x-x_{2}), & x \geq x_{2}. \end{cases}\label{SpecificFunctions}\\
c(x)&=&c(x,0)=c=\textrm{constant},\nonumber
\end{eqnarray}
where $\kappa,A$ and $B$ are positive constants and $0<x_{1}<x_{2}$.
Of course $\kappa$ is the jump of the function $v(x)$ at $x=x_{1}$.
Moreover, $v(x)$ has a corner point at $x=x_{2}$.

The process $X^{\epsilon}_{t}$ that is governed by
the operator $\epsilon D_{v}D_{u}$ is a time changed Wiener process with delay at $x=x_{1}$.

We shall derive
the position of the wave front for this simple case.

It is clear that inside the half lines and line segments $\lbrace
x<x_{1}\rbrace, \lbrace x_{1}<x<x_{2}\rbrace$ and $\lbrace
x>x_{2}\rbrace$ the process $X^{\epsilon}_{t}$ that is governed by
the operator $\epsilon D_{v}D_{u}$ behaves like a standard Wiener
process. Hence, the extremals $\phi$ of the variational problem
(\ref{DefinitionOfWIntro}) for the functional
$R_{0t}(\phi)=ct-S_{0t}(\phi)$ are line segments. Moreover, clearly,
condition (N) holds.

The position of the wave front (interface) for any  couple $(t,x)$
is given by the equation $W(t,x)=0$. Let $t_{*}=t_{*}(x)$ satisfy
the equation $W(t_{*}(x),x)=0$. Such a $t_{*}(x)$ is defined in a
unique way.

For $x\in[0,x_{1})$ the position of the wave front is
\begin{equation}
W(t_{*},x)=0\Rightarrow ct_{*}-\frac{A}{4}\frac{x^{2}}{t_{*}}=0
\Rightarrow t_{*}(x)=\sqrt{\frac{A}{4c}}x
\label{PositionOfWaveFront1}
\end{equation}

For $x\in[x_{1},x_{2})$ the
position of the wave front is as follows. Assume that $0\leq\mu_{0}\leq\mu_{1}\leq
t_{*}$  and that for $t\in[0,\mu_{0}]$ and for $t\in[\mu_{1},t_{*}]$ the function $\phi$ is linear. For $t\in[\mu_{0},\mu_{1}]$ we assume that $\phi(t)=x_{1}$. Straightforward algebra shows that

\begin{eqnarray}
\sigma_{\phi}(t)&=&\begin{cases}\frac{2}{A}t, & 0\leq t\leq \mu_{0} \\
              \frac{2}{A}\mu_{0}, & \mu_{0}\leq t < \mu_{1} \\
              \frac{2}{A}(t-\mu_{1}+\mu_{0}), & \mu_{1}\leq t\leq t_{*}. \end{cases}\nonumber\\
\phi(\gamma_{\phi}(t))&=&\begin{cases}\frac{x_{1}-x}{\mu_{0}}\frac{A}{2}t+x, & 0\leq t\leq \frac{2}{A}\mu_{0} \\
                            -\frac{x_{1}}{t_{*}-\mu_{1}}(\frac{A}{2}t+\mu_{1}-\mu_{0})+\frac{x_{1}t_{*}}{t_{*}-\mu_{1}}, & \frac{2}{A}\mu_{0}\leq t\leq \frac{2}{A}(t_{*}-\mu_{1}+\mu_{0}). \end{cases}\nonumber
\end{eqnarray}
Therefore, we get
\begin{eqnarray}
W(t_{*},x)=0&\Rightarrow& ct_{*}-\inf_{0\leq\mu_{0}\leq\mu_{1}\leq
t_{*}}\{\frac{A}{4}\frac{(x-x_{1})^{2}}{\mu_{0}}+\frac{A}{4}\frac{x_{1}^{2}}{t_{*}-\mu_{1}}\}=0\nonumber\\
&\Rightarrow& t_{*}(x)=\sqrt{\frac{A}{4c}}x
\label{PositionOfWaveFront2}
\end{eqnarray}
In a similar fashion one can show that for $x\in[x_{2},\infty)$ the
position of the wave front is given by
\begin{eqnarray}
W(t_{*},x)=0&\Rightarrow& ct_{*}-\inf_{0\leq\mu_{0}\leq\mu_{1}\leq
t_{*}}\lbrace \frac{A+B}{4}\frac{(x-x_{2})^{2}}{\mu_{0}}
+\frac{A}{4}\frac{x^{2}_{2}}{t_{*}-\mu_{1}}\rbrace=0\nonumber\\
&\Rightarrow& t_{*}(x)=\sqrt{\frac{A}{4c}} [x_{2}+\frac{\sqrt{A+B}}{\sqrt{A}}(x-x_{2})].
\label{PositionOfWaveFront3}
\end{eqnarray}

We make the following remarks.

\begin{rem}
When $u(x)$ and $v(x)$ are smooth linear functions, say for example $u(x)=x$ and $v(x)=Ax$, then the $D_{v}D_{u}$ operator corresponds to a standard Wiener process with a diffusion coefficient that depends on the slopes of $u$ and $v$. In particular, for the case  $u(x)=x$ and $v(x)=Ax$, we have that $D_{v}D_{u}=\frac{1}{A}\frac{d^{2}}{dx^{2}}$, the diffusion coefficient is $\sqrt{\frac{2}{A}}$ and, as it is well known (see for example \cite{KPP} and \cite{F1}), the front travels with constant K-P-P speed $\sqrt{\frac{4c}{A}}$. However, as we can see from equations (\ref{PositionOfWaveFront1})-(\ref{PositionOfWaveFront2}) and (\ref{PositionOfWaveFront3}), the corner points of $u$ and $v$ functions cause a change in the speed of propagation of the front. In particular, in the example considered above, the wave front travels with speed $\sqrt{\frac{4c}{A}}$ for $x<x_{2}$ and with speed  $\sqrt{\frac{4c}{A+B}}$ for $x>x_{2}$. Namely, the speed of propagation is different for different areas of the semi-axis $\lbrace x:x>0\rbrace$.
\begin{flushright}
$\square$
\end{flushright}
\end{rem}

\begin{rem}
Moreover, a careful inspection of the calculations above shows the quite remarkable result that even
though the function $v$ has a discontinuity at the point $x=x_{1}$,
the action functional, evaluated at the function $\phi$ that attains the supremum of (\ref{DefinitionOfWIntro}), does not see this. This implies that the discontinuity of $v$ at the point $x=x_{1}$ does not affect the propagation of the wave front and, in particular, it does not cause delay of the wave front. By delay of the wave front we mean the situation where the
wave front stays on a particular point for a positive amount of
time. At first sight, this is counterintuitive
since one would expect the wave front to experience delay at this
point because the underlying process has delay at $x=x_{1}$.  However, as we saw, this is not true for this case. See the next section
for some more detailed discussion on this.
\begin{flushright}
$\square$
\end{flushright}
\end{rem}

\begin{rem}
One may also assume that $c(x)$ is not homogeneous in $x$. For example, one may suppose that
$c(x)=c_{1}>0$ for $x<x^{*}$ and $c(x)=c_{2}>0$ for $x>x^{*}$, where $0<c_{1}<c_{2}$ are constants and $x^{*}$ is some point on the positive x-axis. It is well known, \cite{F1}, that in the case of standard reaction-diffusion equations, i.e. when the operator is the standard second order elliptic operator, the condition $c_{2}>2c_{1}$ leads to  jumps of the wave front (for more details see \cite{F1}). It is easy to see that the aforementioned effects carry out in the case of generalized reaction diffusions as well.
\begin{flushright}
$\square$
\end{flushright}
\end{rem}

\begin{rem}
In this example we assumed that $u(x)=x$ just for brevity. Of course, one could also assume that $u$ has corner points. Then the phenomena that one observes are similar to the ones described above. Moreover, one can easily extend the aforementioned to the case where $u$ and $v$ have more than one non smoothness points.
\begin{flushright}
$\square$
\end{flushright}
\end{rem}

These complete the study of wave front propagation for piecewise linear functions $u$
and $v$.

\section{Concluding remarks}
In this paper we considered the large deviations principle for a
large class of one dimensional strong Markov processes that are
continuous with probability one. These processes were uniquely
characterized by Feller \cite{Feller} by a generalized second order
differential operator $D_{v}D_{u}$ and its domain of definition. We
derived the action functional for a strong Markov process
$X^{\epsilon}_{t}$ with operator $\epsilon D_{v}D_{u}$. Of course,
such a process can be derived by the process $X_{t}$ that is
governed by the operator $D_{v}D_{u}$ through a time change
$t\rightarrow \epsilon t$, i.e. $X^{\epsilon}_{t}=X_{\epsilon t}$.
We also considered reaction diffusion equations whose operator is a
$D_{v}D_{u}$ operator and studied the problem of wave front
propagation for K-P-P type generalized reaction diffusion equations
in a simple but intuitive setting.

However, the following questions arise naturally.
\begin{enumerate}
\item The process that we considered is governed by an operator of
the form $\epsilon D_{v}D_{u}$, i.e. the $\epsilon$ multiplies the
operator and the functions $v$ and $u$ are independent of
$\epsilon$. A natural question arises.  What kind of dependence of
the functions $v$ and $u$ on $\epsilon$ would guarantee a large
deviations principle for the resulting process $?$ Related to the
latter question is also the following. How could one incorporate the
drift in the action functional in this general setting $?$ In  other words, what is the right formulation of the problem, which would include the usual case (\ref{ActionFunctionalSmooth}), with the drift term $b(\cdot)$ present, as a special case $?$
\item What other phenomena could one observe due to the non-smoothness points of $u$ and $v$ functions $?$ For example, in what scenario would the wave front have delay at particular
points $?$ It is natural to expect delay at points of
discontinuity of the function $v$, since at these points the corresponding process has delay. However, as we saw in the
previous section, the simple situation where $v$ has finitely many discontinuity points and it is independent of $\epsilon$ does not give delay of the front.  The same is true even if we assume that $v$ is discontinuous at every integer point for example. This is because the front has the ``tendency'' to propagate forward and this scenario is not sufficient to ``slow down'' the front at these points. One, probably, needs to consider a more involved situation where $u$ and/or $v$ functions would also depend on $\epsilon$.
\end{enumerate}

We plan to address these questions in a future work.
\section{Acknowledgments}
I would like to thank Professor Mark Freidlin for
our valuable discussions. I would also like to thank Professors Manoussos Grillakis and Sandra Cerrai for their interest in this
work and helpful discussions. 
Lastly, I would like to thank the anonymous referee
for the constructive comments and suggestions that greatly improved the paper.

\end{document}